\documentclass[12pt,reqno,oneside]{amsart}

\usepackage{latexsym,amssymb,amsmath,amsthm,amsfonts,
caption,subcaption,tikz,comment,mathtools}
\usepackage{dirtytalk,times,enumitem}
\usepackage{pgfplots}
\usetikzlibrary{
intersections, arrows.meta,
automata,er,calc,
backgrounds,
mindmap,folding,
patterns,
decorations.markings,
fit,decorations,
shapes,matrix,
positioning,
shapes.geometric,3d,
arrows,through, graphs, graphs.standard,
pgfplots.fillbetween
}

\usepackage[alphabetic]{amsrefs}

\usepackage[colorlinks]{hyperref}
\usepackage[capitalise, noabbrev]{cleveref}
\hypersetup{
  linkcolor=[rgb]{0.2,0.2,0.6}, 
  citecolor=[rgb]{0.2, 0.6, 0.2}, 
  urlcolor=[rgb]{0.6, 0.2, 0.2} 
}

\usepackage{blkarray}
\usepackage{kbordermatrix}

\numberwithin{equation}{section}
\usetikzlibrary{positioning}
\usetikzlibrary{arrows}

\newtheorem{theorem}{Theorem}[section]
\newtheorem{lemma}[theorem]{Lemma}
\newtheorem{proposition}[theorem]{Proposition}
\newtheorem{corollary}[theorem]{Corollary}

\newtheorem{question}[theorem]{Question}
\newtheorem{openproblem}[theorem]{Open Problem}

\theoremstyle{definition}
\newtheorem{definition}[theorem]{Definition} 
\newtheorem{remark}[theorem]{Remark}
\newtheorem{example}[theorem]{Example}
\newtheorem{acknowledgement}{Acknowledgement}

\newcommand{\K}{\mathbb{K}}
\newcommand{\C}{\mathbb{C}}
\newcommand{\N}{\mathbb{N}}
\newcommand{\Z}{\mathbb{Z}}
\newcommand{\R}{\mathbb{R}}

\newcommand{\supp}{\text{supp}}
\newcommand{\tuple}[1]{\langle #1\rangle}
\newcommand{\st}{\colon}

\newcommand{\A}{\mathcal{A}}

\DeclareMathOperator{\lk}{link}

\DeclareMathOperator{\htt}{height}

\newcommand{\qand}{\quad \mbox{and} \quad}

\newcommand{\PP}{\mathcal{P}}

\newcommand{\V}{\mathcal{V}}

\newcommand{\bv}{\mathbf{v}}

\newcommand{\flavio}[1]{\textcolor{teal}{Flavio: #1}}
\newcommand{\barbara}[1]{\textcolor{blue}{Barbara: #1}}

\begin{document}
 
\title{Incidence toric ideals and three-point functions}

\author[B. Betti]{Barbara Betti}
\address[Barbara Betti]
{Otto von Guericke University Magdeburg, Faculty of Mathematics (FMA), Institute of Algebra and Geometry (IAG),
Universitätsplatz 2, 39106 Magdeburg,
Germany.}
\email{barbara.betti@ovgu.de}

\author[S. Grate]{Sean Grate}
\address[Sean Grate]
{Department of Mathematics,
Iowa State University,
396 Carver Hall,
Ames, IA 50011,
USA}
\email{sgrate@iastate.edu}

\author[T. Holleben]{Thiago Holleben}
\address[Thiago Holleben]
{Department of Mathematics \& Statistics,
Dalhousie University,
6297 Castine Way,
PO BOX 15000,
Halifax, NS,
Canada B3H 4R2}
\email{hollebenthiago@dal.ca}

\author[F. Salizzoni]{Flavio Salizzoni}
\address[Flavio Salizzoni]
{MPI für Mathematik in den Naturwissenschaften,
Inselstraße 22, 04103, Leipzig,
Germany}
\email{flavio.salizzoni@mis.mpg.de}

\keywords{toric ideals, three-point functions, polytopes, null designs, balanced manifolds}
\subjclass[2020]{05E40, 13F65, 05B05}

 
\begin{abstract}
    We study the ideal of the algebraic relations among $3$-point functions from a combinatorial and topological perspective. We place this problem in the broader setting of incidence toric ideals associated with incidence matrices of $t$-subsets contained in $k$-subsets of $n$ elements. Generators of these ideals admit combinatorial interpretations as null $t$-designs and topological interpretations as balanced orientable normal $d$-pseudomanifolds without boundary. Generators arising from octahedra play a fundamental role in the structure of these ideals.
\end{abstract}

\maketitle



\section{Introduction}

    In this paper, we study toric ideals generated by the relations among the Laurent monomials 
    \begin{equation} \label{equation:3pf}
    C_{i,j,k}=\frac{1}{p_{i,j}p_{j,k}p_{i,k}},
    \end{equation}
    for $1\leq i<j<k\leq n$, where the $p_{i,j}$ are pairwise invariants arising from a physics setting that we now briefly introduce.  
    
    The momenta of $n$ particles, considered as points in $\mathbb{C}^d$, are given by vectors $p_1,\dots,p_n\in~\C^d$, where $p_i=(x_{1,i},\dots,x_{d,i})$. Endowing $\C^d$ with the Lorentzian inner product, we obtain the generators of the invariant ring of $\C[x_{s,t}]$ under the $\mathrm{SO}(1,d-1)$-action as:
\[      
    p_{i,j} \colon= p_i\cdot p_j=- x_{1,i}x_{1,j}+x_{2,i}x_{2,j}+\dots+x_{d,i}x_{d,j}, \quad 1\leq i<j\leq n.
\]
    Here, $\mathrm{SO}(1,d-1)$ denotes the group of $d\times d$ special orthogonal matrices that preserve the inner product \cite[Section~12]{weyl1946classical}. The pairwise inner products $p_{i,j}$ are known as \emph{Mandelstam invariants} in physics. For fixed $n$, the ideals $\tilde{I}_{n,d}\subseteq \C[c_{i,j,k}]$ of algebraic relations among the Laurent monomials in \cref{equation:3pf} form a stabilizing descending chain
    \[ \tilde{I}_{n,1} \supseteq \tilde{I}_{n,2} \supseteq \cdots \supseteq \tilde{I}_{n,n}=\tilde{I}_{n,n+1}=\cdots \]
    The $p_{i,j}$ are algebraically independent if and only if $n\leq d$ (see \cite[Theorem~2.5]{rajan2024kinematic}). In this case, the ideal $\tilde I_{n,d}$ is toric and it can be computed by looking at the relations among the cubic monomials $c_{i,j,k}\coloneqq C_{i,j,k}^{-1}$. With this interpretation, $\tilde I_{n,d}$ is an ideal in $\C[p_{i,j}:1\leq i< j \leq n]$.
    
    These ideals naturally arise in the study of three-point functions in Conformal Field Theory (CFT). If the particles are massless, the norms squared of the momenta are zero: $p_{i,i}=0$. 
    In CFT, the massless condition allows to not break the symmetry by introducing a preferred length scale. 
    Correlations between quantum field operators at $n$ spacetime points are called $n$-point functions and serve as observables in CFT.
    Higher-point functions are more complicated, but in many cases they can be reduced to $2$- and $3$-point functions through operator product expansions (OPE). Equation $2.38$ in \cite{rychkov2017epfl} gives a first example of $3$-point function as the correlator $\langle \varphi(p_i)\varphi(p_j)\varphi(p_k)\rangle$ of three primary scalar fields of dimension $\Delta_i,\Delta_j$ and $\Delta_k$: 
    \begin{equation}\label{equation:3pointfunctionprimaryfield}
    C_{i,j,k}:=\langle \varphi(p_i)\varphi(p_j)\varphi(p_k)\rangle= c\cdot \left(p_{i,j}^{\frac{\Delta_i+\Delta_j-\Delta_k}{2}}\cdot p_{i,k}^{\frac{\Delta_i+\Delta_k-\Delta_j}{2}}\cdot p_{j,k}^{\frac{\Delta_j+\Delta_k-\Delta_i}{2}}\right) ^{-1}, \ c\in\C.   
    \end{equation}

    Assuming $\Delta_1=\dots=\Delta_n=2$, Equation~\eqref{equation:3pointfunctionprimaryfield} simplifies, up to a constant, to the Laurent monomials in \cref{equation:3pf}.  
    In this paper, we study such simplified 3-point functions from an algebraic and combinatorial point of view.

    We discuss the toric case, namely when $n\leq d$, within the more general framework of \emph{incidence toric ideals} (\cref{def:inctorideal}). These are defined as the toric ideals associated with the incidence matrices of the $(t,k)$-incidence complexes of the $(n-1)$-dimensional simplex. Such matrices encode containments of $t$-subsets inside $k$-subsets of $n$ elements. We show they have maximal rank by using foundational results due to Stanley on the theory of Lefschetz properties of artinian algebras (\cref{prop: A_nkt full rank}). Minimal generators of the incidence toric ideal $I_{n,k,t}$ arise from connections with \emph{null $t$-designs}, that are balanced functions used in combinatorial design theory. This connection allows to bound from below the degrees of generators and to show neighbourliness of the associated rational polytope (\cref{theorem:neighborly}).
    
    In \Cref{sec: balanced manifolds} we move to the topological point of view. We relate minimal generators in a Graver basis of $I_{n,k,k-1}$ to balanced orientable normal $d$-pseudomanifolds without boundary. Notice that this class of incidence toric ideal still includes our main motivating class of ideals of relations of $3$-point functions, as $\tilde{I}_{n,n}=I_{n,3,2}$. Binomials coming from balanced triangulations of (hyper)octahedra carry most of the geometric structure of $I_{n,k,t}$, as, up to saturation, they generate the same ideal. Moreover, the exponents of these binomials arise in representation theory as Specht polynomials (see~\cite[Section 9.3]{lefschetzbook}).
    We discuss this relation in \cref{ss:specht}. 

    In \Cref{sec: three point functions} we discuss the first  non-toric case of the ideals $\tilde{I}_{n,d}$. We show that $\tilde{I}_{n,n-1}$ is obtained by adding a principal ideal to $\tilde{I}_{n,n}$ and saturating with respect to the $c_{i,j,k}$.

\section{Preliminaries}
\subsection{Toric ideals}


    Toric ideals are a central bridge between commutative algebra and combinatorics. They are binomial prime ideals arising as kernel of (Laurent) monomial maps. Let $\mathbb{K}$ be a field and $A=({\bf a}_1 \ \cdots \ {\bf a}_n) \in \Z^{d\times n}$ be the matrix defining the $\mathbb{K}$-algebra homomorphism 
    \[ \pi:  \mathbb{K}[x_1,\dots,x_n]\rightarrow \mathbb{K}[t_1^\pm,\dots, t_d^\pm], \ x_i\mapsto t_i^{{\bf a}_i}.\]
    The toric ideal $I_A$ associated with $A$ is the kernel of $\pi$. Algebraic properties of $I_A$ are encoded by $A$ and the polytope $\mathcal{P}={\rm conv}(A)\subseteq \R^d$, obtained as the convex hull of ${\bf a}_1,\dots,{\bf a}_n\in \R^d$. We recall some well-known facts and we refer to \cite[Chapter~4]{sturmfelsgroebner} for further details. The Krull dimension of the quotient by $I_A$ is the dimension of the $\Z$-lattice spanned by $A$, that equals the rank of $A$. Moreover, if $I_A$ is homogeneous, $\mathcal{V}(I_A)$ is a projective variety whose degree is the normalized volume of $\mathcal{P}$. 
    Binomials of the form $x^{u_+}-x^{u_-}$, where $u=u_+-u_-\in {\rm ker}_\Z(A)$ form a $\mathbb{K}$-basis of $I_A$. Such a binomial is \emph{primitive} if there is no other vector $v= v_+-v_-\in {\rm ker}_\Z(A)$ such that $v_+< u_+$ and $u_+<u_-$ coordinate-wise. The set of all primitive binomials forms the \emph{Graver basis} of $I_A$.
    Any finite generating set of binomials for $I_A$ extracted from the binomials in $\{x^{u_+}-x^{u_-}: \ u_+-u_-\in {\rm ker}_\Z(A)\}$ forms a \emph{Markov basis
    } for $I_A$ \cite[Theorem~3.1]{SturmfelsDiaconis}. 

\begin{definition}
    A finite set $\mathcal{M}\subseteq  {\rm ker}_\Z(A)$ is a \emph{Markov basis} of $I_A$ if for every pair $u,v \in \N^n$ with $A(u-v)=0$, there are $b_1,\dots,b_r\in \mathcal{M}$ such that $u=v+b_1+\dots+b_r$ and the partial sums $v+b_1+\cdots+b_c$ have non-negative coordinates for every $c=1,\dots,r$.  
\end{definition}

    A significant class of toric ideals appears from simple graphs. The \emph{edge ring} of a graph $G$ is the toric algebra generated by monomials $x_e= x_ix_j$ such that $e=(i,j)$ is an edge of the graph. The defining ideal of the edge ring is the toric ideal $I_G$, that is the kernel of the homomorphism
    \[ \mathbb{K}[y_e: e\in E(G)]\rightarrow \mathbb{K}[x_e: e\in E(G)], \ y_e\mapsto x_e. \]
    A finite set of binomial generators is given by primitive even closed walks \cite[Proposition~3.1]{Villarreal01011995}. An even closed walk $w=\{e_1,\dots,e_{2k}\}\subseteq E(G)$ is a sequence of an even number of edges where each consecutive pair $e_i, e_{i+1}$ shares a common vertex, including the first and the last edge. A primitive even closed walk is an even closed walk that cannot be decomposed into smaller even closed walks. A (primitive) even closed walk identifies the (primitive) binomial 
    \[f_w = y_{e_1}y_{e_3}\cdots y_{e_{2k-1}} - y_{e_2}y_{e_4}\cdots y_{e_{2k}}\in I_G.\]
    
    In \cite{PetrovicStasi} the authors propose a natural generalization of even closed walks from graphs to hypergraphs under the name of \emph{monomial walks}. 
    Edges are replaced by hyperedges and the alternating structure of an even closed walk is reformulated in terms of balanced edge multisets. 
    Hyperedges are partitioned into two parts where every vertex appears with the same multiplicity in both parts.
    As in the graphical case, monomial walks correspond to binomials in the associated toric ideal. 
    However, primitive binomials arising from monomial walks in hypergraphs do not necessarily have squarefree support as balancing hyperedges allows the use of the same vertex repeatedly.
    We adopt the setting of \cite{PetrovicStasi} and reinterpret it in the language of simplicial complexes. 
    This translation allows us to treat monomial walks as combinatorial objects supported on facets of a simplicial complex, extending the correspondence between balanced walks and binomials in toric ideals beyond the graphical case. 
    



\subsection{Balanced simplicial complexes}

    A \emph{simplicial complex} on $n$ vertices is a collection $\Delta$ of subsets of $[n]$, called \emph{faces}, that is closed under inclusion. 
    A face $F$ has dimension $\dim(F)=|F|-1$ and the dimension of a simplicial complex $\Delta$ is the maximum of the dimensions of its faces. 
    An inclusion-maximal face is called a \emph{facet} and a face of dimension zero is a \emph{vertex}. A simplicial complex is \emph{pure} if all the facets have the same dimension. 
    We write $\Delta=\langle F_1,\dots, F_s\rangle$ for the simplicial complex generated by the facets $F_1,\dots, F_s$.
    The \emph{incidence matrix} of $\Delta = \tuple{F_1, \dots, F_s}$ is the $s\times n$ matrix $M(\Delta)$ with entry $1$ in position $ij$ if $i\in F_j$ and $0$ otherwise. 
    The monomial map defined by the transpose of the incidence matrix of $\Delta$ induces a toric ideal associated to $\Delta$, which we will consider in later sections.

\begin{definition}[\cite{Stanley1979BalancedCC}]
    A simplicial complex $\Delta$ on $n$ vertices of dimension $d-1$ is \emph{balanced} if there exists a partition $[n]= V_1\sqcup \dots \sqcup V_d$ such that $|F\cap V_i|\leq 1$ for every face $F\in \Delta$. 
\end{definition}
    Equivalent characterizations are that the underlying graph of $\Delta$ is $d$-colorable and that each facet contains exactly one vertex of each color.
    In \Cref{sec: balanced manifolds} we study the balanced property for the following class of simplicial complexes. 

    A $d$-dimensional pure simplicial complex $\Delta$ is a~\emph{pseudomanifold} if the following hold:

    \begin{enumerate}
        \item $\Delta$ is strongly connected, that is, for every pair of facets $F, F'$ of $\Delta$, there exists a sequence of facets $F_0 = F, F_1, \dots, F_s = F'$ such that $F_i \cap F_{i + 1}$ is a $(d-1)$-dimensional face of $\Delta$ for every $0 \leq i \leq s-1$.
        \item Every $(d-1)$-dimensional face of $\Delta$ is contained in at most $2$ facets of $\Delta$.
    \end{enumerate}

    For a pseudomanifold $\Delta$ of dimension $d$, the simplicial complex $\partial \Delta$ whose facets are the $(d-1)$-dimensional faces of $\Delta$ contained in exactly one facet, is called the~\emph{boundary complex} of $\Delta$. If $\Delta$ is a pseudomanifold such that $\partial \Delta = \emptyset$, we say $\Delta$ is a pseudomanifold~\emph{without boundary}. A pseudomanifold without boundary $\Delta$ is said to be~\emph{orientable} if $\tilde H_d(\Delta; \Z) = \Z$. 

    \begin{example}
    Orientable pseudomanifolds without boundary include the class of simplicial complexes with a geometric realization homeomorphic to a $d$-dimensional sphere. Such complexes are called~\emph{simplicial spheres}.
    The standard example of a balanced $(d-1)$-sphere is the boundary of the $d$-dimensional crosspolytope. This is the polytope in $\R^d$ with $2d$ vertices obtained as the convex hull of $\{ \pm e_i \, \colon \, i=1,\dots, d\}$, where $e_i$ is the $i$-th element of the standard basis. In dimensions $d=2$ and $d=3$, the boundary of the crosspolytope are, respectively, a square and an octahedron, which are a balanced $1$-sphere and a balanced $2$-sphere (see \cref{figure:cross-polytope}).

    {\centering
    \begin{figure}[h!]
    \begin{tikzpicture}[scale=1.5]
            \node[draw,circle,fill=black,scale=0.2,label={[label distance=-1pt]210:$e_1$}] at (xyz spherical cs:radius=1.41,latitude=135,longitude=90) (1) {$1$};
            \node[draw,circle,fill=black,scale=0.2,label={-45:$-e_2$}] at (xyz spherical cs:radius=1.41,latitude=-135,longitude=90) (2) {$2$};
            \node[draw,circle,fill=black,scale=0.2,label={[label distance=-1pt]0:$-e_1$}] at (xyz spherical cs:radius=1.41,latitude=-45,longitude=90) (3) {$3$};
            \node[draw,circle,fill=black,scale=0.2,label={[label distance=-5pt]-30:$e_2$}] at (xyz spherical cs:radius=1.41,latitude=45,longitude=90) (4) {$4$};
            \node[draw,circle,fill=black,scale=0.2,label={[label distance=-1pt]150:$e_3$}] at (xyz spherical cs:radius=1.82) (5) {$5$};
            \node[draw,circle,fill=black,scale=0.2,label={[label distance=-1pt]210:$-e_3$}] at (xyz spherical cs:radius=1.82,latitude=180) (6) {$6$};

            \begin{scope}[on background layer]
                \path[fill=blue,opacity=0.25] (3.center) to (4.center) to (5.center) to (3.center);
                \path[fill=red,opacity=0.25] (4.center) to (1.center) to (5.center) to (4.center);
                \path[fill=red,opacity=0.25] (3.center) to (4.center) to (6.center) to (3.center);
                \path[fill=blue,opacity=0.25] (4.center) to (1.center) to (6.center) to (4.center);
            \end{scope}
            
            \draw[very thick, dashed] (1) -- (2) -- (3);
            \draw[very thick] (3) -- (4) -- (1);
            \draw[very thick] (1) -- (5);
            \draw[very thick, dashed] (5) -- (2);
            \draw[very thick] (3) -- (5) -- (4);
            \draw[very thick] (1) -- (6);
            \draw[very thick, dashed] (6) -- (2);
            \draw[very thick] (3) -- (6) -- (4);
        \end{tikzpicture}
        \caption{The crosspolytope for $d=3$: an octahedron, obtained as the convex hull of $\pm e_1$, $\pm e_2$, and $\pm e_3$. The alternating red and blue coloring reflects that its boundary is a balanced $2$-sphere (see~\cref{l:bipartitedual}).}
        \label{figure:cross-polytope}
        \end{figure}
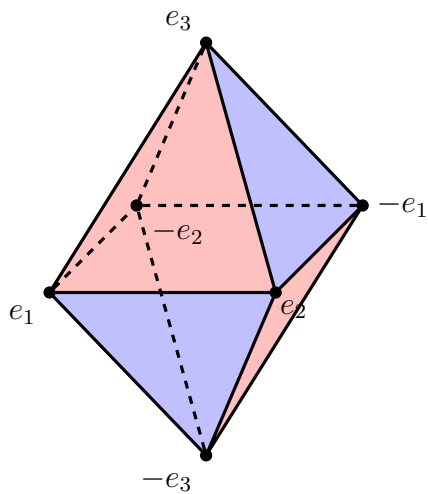}
\end{example}

The~\emph{link} of a simplicial complex $\Delta$ with respect to a face $\sigma \in \Delta$, is the simplicial complex $\lk_\Delta(\sigma) = \{\tau \st \tau \cap \sigma = \emptyset \qand \tau \cup \sigma \in \Delta\}$. The links of a simplicial complex $\Delta$ contain information about the local properties of $\Delta$. A $d$-dimensional pseudomanifold $\Delta$ is~\emph{normal} if $\lk_\Delta(\sigma)$ is connected for every face $\sigma$ of $\Delta$ of dimension $k \leq d - 2$. The main examples of normal pseudomanifolds are simplicial spheres. A non-normal (orientable) pseudomanifold is any triangulation of the~\emph{pinched torus} (\Cref{figure:pinched torus}). Triangulations of the pinched torus can be obtained from a triangulation of a sphere after identifying two vertices that do not appear together in the link of any nonempty face. Since the link of the identified vertex is not connected, the triangulation is not normal.

{\centering

\begin{figure}[h!]
\includegraphics[scale = 0.3]{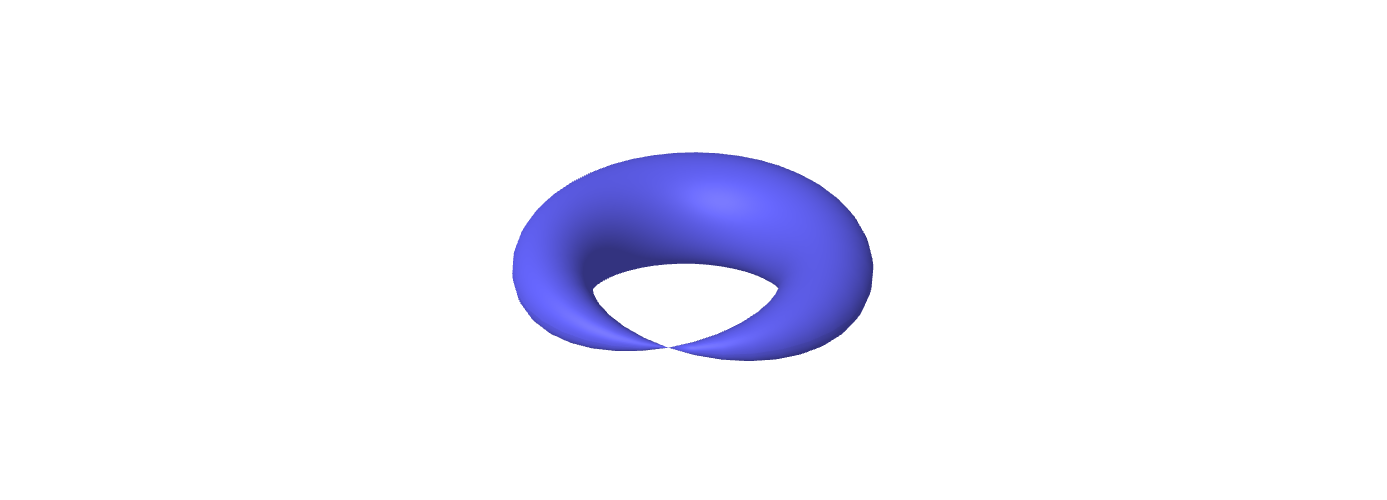}
\caption{The pinched torus: a topological space whose triangulations are non-normal orientable pseudomanifolds.}
\label{figure:pinched torus} 
\end{figure}}

Finally, given a $d$-dimensional pseudomanifold $\Delta = \tuple{F_1, \dots, F_s}$, the~\emph{facet-ridge graph} $G(\Delta)$ is the graph on vertex set $\{F_1, \dots, F_s\}$ and $\{F_i, F_j\}$ is an edge if and only if $F_i \cap F_j$ is a~\emph{ridge}, that is, a face of $\Delta$ dimension $d-1$.

\subsection{Lefschetz properties}
    A standard graded Artinian $\mathbb{K}$-algebra $A$ satisfies the \emph{weak Lefschetz property} (\emph{WLP}) if there is a general linear form $L \in A_1$ such that the induced multiplication maps $\times L: A_i \to A_{i + 1}$ have full rank for every $i$; such a linear form $L$ is called a \emph{weak Lefschetz element of $A$}. 
    If the iterated multiplication maps $\times L^j: A_i \to A_{i + j}$ have full rank for every $i,j>0$, then~$A$ satisfies the \emph{strong Lefschetz property} (\emph{SLP}) and we say that $L$ is a \emph{Lefschetz element} for $A$.

    Lefschetz properties serve as algebraic incarnations of the classical hard Lefschetz theorem.
    As one of the original applications of these Lefschetz-type theorems to problems in enumerative combinatorics, the proof of the necessity in the $g$-theorem invokes the hard Lefschetz Theorem~\cite{gtheorem}.
    Using this machinery, Stanley shows the Bruhat ordering of certain Weyl groups have the Sperner property~\cite{S80}.
    This amounts to showing that, in characteristic zero, Artinian $\mathbb{K}$-algebras obtained as quotients by monomial complete intersections have the SLP.
\begin{theorem}[{\cite{S80},\cite[Theorem~3.35]{lefschetzbook}}]\label{t:monomialslp}
    For any $d_1, \dots, d_n\in \N^*$ and any field $\K$ of characteristic $0$, the algebra $ \K[x_1,\dots,x_n]/(x_1^{d_1},\dots, x_n^{d_n})$ has the strong Lefschetz property.
\end{theorem}

    In general, exhibiting a Lefschetz element is a difficult task. However, if $A$ is an Artinian $\mathbb{K}$-algebra defined by the quotient by a monomial ideal, then the SLP and the WLP are entirely determined by a specific linear form.
    This fact will be used implicitly throughout.

\begin{lemma}[{\cite[Proposition~2.2]{MMN2011}}]
    Let $I \subset R = \K[x_1, \dots, x_n]$ be an Artinian monomial ideal, where $\K$ is a field of characteristic zero. 
    Then $A = R/I$ has the SLP (resp. WLP) if and only if $L = x_1 + \dots + x_n$ is a (weak) Lefschetz element for $A$. 
\end{lemma}

\section{Incidence toric ideals}

    
    Let $\Delta$ be a pure simplicial complex on $[n]$ and $t< k\leq n$ integers. The \emph{$(t,k)$-incidence complex} $H_\Delta(t,k)$ is the simplicial complex with vertex set the $t$-faces of $\Delta$ and $\sigma$ is a facet if and only if it is the collection of $t$-faces of a $k$-face of $\Delta$. 
    In other words, facets of $H_\Delta(t,k)$ are collections of $t$-faces of $\Delta$ using the same $k+1$ vertices. 
    By construction, $H_\Delta(t,k)$ is pure and \[\dim(H_\Delta(t,k))=\binom{k+1}{t+1}-1.\]
    In order to take into account all the $t$-subsets in every $k$-subsets, we consider the $(t-1,k-1)$-incidence complex $H_{\Delta_{n-1}}(t-1,k-1)$ of the standard $(n-1)$-simplex $\Delta_{n-1}$. We denote by $\A_{n,k,t}$ the transpose of the incidence matrix of $H_{\Delta{n-1}}(t-1, k-1)$. Note that $\A_{n,k,t}$ is a $\binom{n}{t} \times \binom{n}{k}$ matrix whose rows and columns are indexed, respectively, by the $t$-subsets and $k$-subsets of $[n]$. The toric ideal associated to this matrix is the main object of this paper.
    

\begin{example}\label{example:octahedron}
    Let $\Delta$ be the boundary of an octahedron, pictured in \cref{figure:octahedron}.
    {\centering
    \begin{figure}[h]
    \begin{subfigure}[!h]{0.45\linewidth}
        \centering
        \begin{tikzpicture}[scale=1.6]
            \node[draw,circle,fill=black,scale=0.2,label={[label distance=-1pt]210:$1$}] at (xyz spherical cs:radius=1.41,latitude=135,longitude=90) (1) {$1$};
            \node[draw,circle,fill=black,scale=0.2,label={180:$3$}] at (xyz spherical cs:radius=1.41,latitude=-135,longitude=90) (2) {$2$};
            \node[draw,circle,fill=black,scale=0.2,label={[label distance=-1pt]0:$2$}] at (xyz spherical cs:radius=1.41,latitude=-45,longitude=90) (3) {$3$};
            \node[draw,circle,fill=black,scale=0.2,label={[label distance=-5pt]-30:$4$}] at (xyz spherical cs:radius=1.41,latitude=45,longitude=90) (4) {$4$};
            \node[draw,circle,fill=black,scale=0.2,label={[label distance=-1pt]150:$5$}] at (xyz spherical cs:radius=1.82) (5) {$5$};
            \node[draw,circle,fill=black,scale=0.2,label={[label distance=-1pt]210:$6$}] at (xyz spherical cs:radius=1.82,latitude=180) (6) {$6$};

            \begin{scope}[on background layer]
                \path[fill=red,opacity=0.25] (1.center) to (4.center) to (5.center) to (1.center);
                \path[fill=blue,opacity=0.25] (3.center) to (4.center) to (5.center) to (3.center);
                \path[fill=blue,opacity=0.25] (1.center) to (4.center) to (6.center) to (1.center);
                \path[fill=red,opacity=0.25] (4.center) to (3.center) to (6.center) to (4.center);
            \end{scope}
            
            \draw[very thick, dashed] (1) -- (2) -- (3);
            \draw[very thick] (3) -- (4) -- (1);
            \draw[very thick] (1) -- (5);
            \draw[very thick, dashed] (5) -- (2);
            \draw[very thick] (3) -- (5) -- (4);
            \draw[very thick] (1) -- (6);
            \draw[very thick, dashed] (6) -- (2);
            \draw[very thick] (3) -- (6) -- (4);
        \end{tikzpicture}
        \caption{The boundary of an octahedron, the simplicial complex of interest in \cref{example:octahedron}.}
        \label{figure:octahedron}   
    \end{subfigure}
    \hfill
    \begin{subfigure}[!h]{0.45\linewidth}
        \centering
        \begin{tikzpicture}[scale=0.95]
            \fill[blue!25] (-1,1) -- (0,2.23) -- (1,1);
            \fill[red!25] (1,1) -- (2.23,0) -- (1,-1);
            \fill[blue!25] (1,-1) -- (0,-2.23) -- (-1,-1);
            \fill[red!25] (-1,-1) -- (-2.23,0) -- (-1,1);
    
            \fill[red!25] (0,2.23) -- (-3,3) -- (3,3);
            \fill[blue!25] (2.23,0) -- (3,3) -- (3,-3);
            \fill[red!25] (0,-2.23) -- (3,-3) -- (-3,-3);
            \fill[blue!25] (-2.23,0) -- (-3,-3) -- (-3,3);
    
            \node[draw,circle,fill=black,scale=0.2,label={[label distance=-1pt]90:$13$}] at (0,2.23) (12) {$12$};
            \node[draw,circle,fill=black,scale=0.2,label={[label distance=-1pt]0:$23$}] at (2.23,0) (23) {$23$};
            \node[draw,circle,fill=black,scale=0.2,label={[label distance=-1pt]-90:$24$}] at (0,-2.23) (34) {$34$};
            \node[draw,circle,fill=black,scale=0.2,label={[label distance=-1pt]180:$14$}] at (-2.23,0) (14) {$14$};
            \node[draw,circle,fill=black,scale=0.2,label={[label distance=-1pt]135:$15$}] at (-1,1) (15) {$15$};
            \node[draw,circle,fill=black,scale=0.2,label={[label distance=-1pt]45:$35$}] at (1,1) (25) {$25$};
            \node[draw,circle,fill=black,scale=0.2,label={[label distance=-1pt]-45:$25$}] at (1,-1) (35) {$35$};
            \node[draw,circle,fill=black,scale=0.2,label={[label distance=-1pt]-135:$45$}] at (-1,-1) (45) {$45$};
            \node[draw,circle,fill=black,scale=0.2,label={[label distance=-1pt]135:$16$}] at (-3,3) (16) {$16$};
            \node[draw,circle,fill=black,scale=0.2,label={[label distance=-1pt]45:$36$}] at (3,3) (26) {$26$};
            \node[draw,circle,fill=black,scale=0.2,label={[label distance=-1pt]-45:$26$}] at (3,-3) (36) {$36$};
            \node[draw,circle,fill=black,scale=0.2,label={[label distance=-1pt]-135:$46$}] at (-3,-3) (46) {$46$};
    
            \filldraw[very thick] (15) -- (12) -- (25) -- (15);
            \filldraw[very thick] (25) -- (23) -- (35) -- (25);
            \filldraw[very thick] (35) -- (34) -- (45) -- (35);
            \filldraw[very thick] (45) -- (14) -- (15) -- (45);
    
            \filldraw[very thick] (12) -- (16) -- (26) -- (12);
            \filldraw[very thick] (23) -- (26) -- (36) -- (23);
            \filldraw[very thick] (34) -- (36) -- (46) -- (34);
            \filldraw[very thick] (14) -- (46) -- (16) -- (14);
        \end{tikzpicture}
        \caption{The incidence complex $H_\Delta(1,2)$ of the boundary of an octahedron from \cref{example:octahedron}.}
        \label{figure:octahedron incidence complex}
        \end{subfigure}
        \caption{}
    \end{figure}}
    Since $\Delta$ is a pure simplicial complex, $H_\Delta(0,1)$ is its underlying graph (the $1$-skeleton) and $H_\Delta(0,\dim{\Delta})=H_\Delta(0,2)$ is the same as $\Delta$.
    The only nontrivial incidence complex of the octahedron $\Delta$ is $H_\Delta(1,2)$, represented in \cref{figure:octahedron incidence complex}. For ease of display, a vertex $\{ i, j \}$ is labeled as $ij$. The facets are shaded blue and red and contain vertices whose labels are the edges of a facet of $\Delta$.
\end{example}


\begin{definition}\label{def:inctorideal}
    Let $\mathcal{A}_{n,k,t}$ be the incidence matrix of $H_{\Delta_n}(t-1,k-1)$. The \emph{incidence toric ideal} $I_{n,k,t}$ is the kernel of the monomial map with exponent matrix $\mathcal{A}_{n,k,t}$:
    \[ \varphi: \K[c_\kappa: \kappa\in \binom{[n]}{k}] \twoheadrightarrow \K[x_\tau: \tau\in \binom{[n]}{t}], \ \  c_\kappa\mapsto \prod_{\tau \subseteq \kappa} x_\tau. \]
\end{definition}
    \begin{example}
    $(n=4,k=3,t=2)$ The monomial map $\varphi$ with exponent matrix $\mathcal{A}_{4,3,2}$ is
    \begin{equation}
       \begin{aligned}
        \varphi: \K[c_{123}, c_{124}, c_{134}, c_{234}]&\rightarrow \K[x_{12},x_{13},x_{14},x_{23},x_{24},x_{34}] \\ 
        c_{123}\mapsto x_{12}x_{13}x_{23}, \quad 
        c_{124} \mapsto x_{12}x_{14}x_{24}&, \quad
        c_{134}\mapsto x_{13}x_{14}x_{34}, \quad 
        c_{234}\mapsto x_{23}x_{24}x_{34}.  
    \end{aligned} 
    \end{equation} 
    We index the rows and columns of $\mathcal{A}_{4,3,2}$ by subsets of two and three elements of $\{1,2,3,4\}$:
\begin{equation} \label{eq: matrix A_4,3,2}
\small
   \kbordermatrix{
        & 123 & 124 & 134   & 234 \\
    12  &1    & 1   & 0     & 0  \\
    13  &1    & 0   & 1     & 0  \\
    23  &1    & 0   & 0     & 1  \\
    14  &0    & 1   & 1     & 0  \\
    24  &0    & 1   & 0     & 1  \\
    34  &0    & 0   & 1     & 1  \\
    } =\mathcal{A}_{4,3,2}.
    \end{equation}
\end{example}

    We first observe that $I_{n,k,t}$ is homogeneous and it defines an irreducible projective toric variety $\mathcal{V}(I_{n,k,t})\subseteq \mathbb{P}^{\binom{n}{k}-1}$. Indeed, since columns of $\mathcal{A}_{n,k,t}$ have $\binom{k}{t}$ entries equal to $1$ and $0$ elsewhere, the vector $(1,\dots,1)$ is in the rowspan of $\mathcal{A}_{n,k,t}$. If $t=k$, then $\mathcal{A}_{n,k,t}$ is the identity matrix. 

    We show that $\mathcal{A}_{n,k,t}$ has full rank by realizing it as the transpose of the multiplication matrix by the linear form $x_1+\dots +x_n$ in a monomial complete intersection, that always has the SLP. 

\begin{proposition} \label{prop: A_nkt full rank}
    Let $\K$ be a field of characteristic zero and let $L = x_1 + \dots + x_n $ be a linear form in the graded Artinian $\K$-algebra $ A= \frac{\K[x_1, \dots, x_n]}{(x_1^2, \dots, x_n^2)}$. Then, up to a constant, the multiplication~map 
    \[ \times L^{k-t}: A_{t}  \to A_{k} \]
    is represented by the transpose of the incidence matrix $\mathcal{A}_{n,k,t}$. In particular, $\mathcal{A}_{n,k,t}$ has full rank:
    \[ {\rm rank}(\mathcal{A}_{n,k,t}) = \min\left( \binom{n}{t}, \binom{n}{k} \right).\]
\end{proposition}

\begin{proof}
    Let $\Delta$ be a $(n-1)$-dimensional simplex. By~\cite[Proposition 3.7]{H2025} the incidence matrix of $H_\Delta(t-1,k-1)$ is (up to a constant) the matrix given by $\times L^{k - t}: A_t \to A_{k}$. By \Cref{t:monomialslp}, $A$ has the strong Lefschetz property, that is, $\mathcal{A}_{n,k,t}$ is full rank for every $t,k$.
\end{proof}

    From \cref{prop: A_nkt full rank}, it follows that the toric ideal $I_{n,k,t}$ is nonzero if and only if $\binom{n}{t}<\binom{n}{k}$; otherwise, the kernel of $\mathcal{A}_{n,k,t}$ is trivial. We therefore assume this inequality holds throughout the remainder of the paper. The first interesting case is shown in the following example.

\begin{example}\label{ex:octahedralquartics}($t=2,k=3,n=6)$
    The incidence toric ideal $I_{6,3,2}$ has $30$ binomial generators. Following the labeling of the octahedron in \Cref{figure:octahedron}, they are $15$ quartics of the form
    \begin{equation} \label{eq: octahedral quartics}
        c_{136}c_{246}c_{145}c_{235} \ -  \ c_{146}c_{236}c_{245}c_{135}
    \end{equation}
    and $15$ sextics of the form
    \[ c_{146}c_{156}c_{236}c_{123}c_{345}c_{245} \ - \ c_{136}c_{126}c_{456}c_{145}c_{234}c_{235}.  \]
    We refer to (\ref{eq: octahedral quartics}) as the \emph{octahedral quartics}, since they correspond to the $15$ relabelings of vertices of an octahedron. The octahedral quartics (\ref{eq: octahedral quartics}) corresponds to a $2$-coloring of the facets. The ideal $I_{6,3,2}$ defines a toric variety of codimension $5$ and degree $162$ in $\mathbb{P}^{19}$.
\end{example}

\begin{corollary}
    Let $n > k > t$ be positive integers such that $\binom{n}{t}<\binom{n}{k}$. Then
    \[\htt (I_{n,k,t}) = \binom{n}{k} - \binom{n}{t}, \quad {\dim}(\V(I_{n,k,t})) = \binom{n}{t}-1.\]
    
\end{corollary}




\subsection{Combinatorial and geometric structure of the polytope} 
    In this section we study the combinatorial properties of the rational polytope $\mathcal{P}_{n,k,t}$ associated with the ideal $I_{n,k,t}$. It is obtained as the convex hull of the columns of $\mathcal{A}_{n,k,t}$. 
    We show that its volume grows with no polynomial bound for suitable choices of $n,k,t$ and there is no polynomial formula in terms of $n,k,t$ for it. We will use the following statement from the theory of Lefschetz properties.

    \begin{theorem}[{\cite[Theorem 2.9]{K2026}}]\label{t:monomialslpcharp}
    Let $A = \frac{\K[x_1,\dots,x_n]}{(x_1^2,\dots, x_n^2)}$, where $\K$ is a field of characteristic $p$ and $L = x_1 + \dots + x_n$. Then for integers $0 \leq t < k \leq n$ the map 
    \[   \times L^{k-t}: A_t \to A_k \]
    has full rank if and only if $p > \min(k, n-t)$.
\end{theorem}

\begin{proposition} \label{prop: volume divided by p}
    The normalized Euclidean volume (with respect to the Euclidean lattice) of $\PP_{n,k,k-1}$ is divisible by every prime $p$ such that 
    $$
        p \leq \min(k, n-k+1).
    $$ 
\end{proposition}

\begin{proof}
    The volume of $\PP_{n,k,k-1}$ is the sum of volumes of the simplices in any of its triangulation, hence it suffices to find a triangulation such that the volume of every simplex is divisible by primes $p \leq \min(k, n - k+1)$. 
    The volume of a simplex $\Delta$ on $\ell+1$ vertices is the absolute value of the determinant of the matrix whose columns are the directions of its edges containing a fixed vertex $\bv_0$. This is the determinant of the $\ell\times \ell$ matrix $M_\Delta$, whose columns are the coordinates of the vertices of $\Delta$, shifted by the first column. Since the gcd of all $\ell$-minors is preserved by row and column operations, the volume of $\Delta$ is equal to the gcd of the $\ell$-minors of the matrix whose columns are the coordinates of vertices of $\Delta$. 
    In particular, if the vertices of $\Delta$ are also vertices of $\PP_{n,k,k-1}$, these minors are nonzero maximal minors of the matrix $\mathcal{A}_{n,k,k-1}$.
    
    Next, note that every polytope can be triangulated without introducing new vertices by simply coning over some vertex. Consequently, if a prime $p$ divides the greatest common divisor of the nonzero maximal minors of $\mathcal{A}_{n,k,k-1}$, it also divides the volume of every simplex in any triangulation mentioned above. To conclude, we consider $[\mathcal{A}_{n,k,k-1}]_p$ to be the matrix $\mathcal{A}_{n,k,k-1}$ modulo $p$. Since $p$ divides the gcd of all $i$-minors of $\mathcal{A}_{n,k,k-1}$ if and only if the rank of $[\mathcal{A}_{n,k,k-1}]_p$ is smaller than $i$, the result follows by \cref{t:monomialslpcharp}.
\end{proof}

\begin{corollary}
    The volume of $\mathcal{P}_{n,k,t}$ is not a polynomial function in $n,k,t$.
\end{corollary}

\begin{proof}
     If $t = k -1$ and $ n= 2k-1$, by \cref{prop: volume divided by p} the volume of $\mathcal{P}_{2k-1,k,k-1}$ is divisible by the product of every prime number smaller than $k$ and such product grows like $e^{\frac{k}{\log(k)} \log(\frac{k}{\log(k)})}$ \cite[Section~22.18]{hardy2008introduction}. Hence, the volume $V(\mathcal{P}_{2k-1,k,k-1})$ grows faster than any polynomial function in $k$. It follows that $V(\mathcal{P}_{n,k,t})$ cannot be expressed as a polynomial function in $n,k,t$.
\end{proof}

We note that the volume being computed in the statements above is not equal to the degree of the variety $\V(I_{n,k,t})$, since we are normalizing the volume with respect to the Euclidean lattice, instead of the lattice induced by $\mathcal{A}_{n,k,t}$. This leaves the following question open.
\begin{question} \label{question: deg I = deg J}
    How to compute efficiently the degree of $I_{n,k,t}$? 
\end{question}



    We employ null $t$-design functions to identify lower dimensional faces of $\mathcal{P}_{n,k,t}$ and to obtain a lower bound for the degrees of generators of $I_{n,k,t}$. We denote by $2^{[n]}$ the powerset of $[n]$.

\begin{definition}
    A function $f:2^{[n]}\rightarrow\mathbb{R}$ is a \emph{null $t$-design} if, for every $X\in 2^{[n]}$ with $\lvert X\rvert\leq t$, the following balance condition holds:
\begin{equation} \label{eq: t-design condition}
	\sum_{X\subseteq F\subseteq [n]} f(F)=0\,.
\end{equation}
    Moreover, we define the support and the positive support of $f$, respectively, as the sets
\begin{equation*}
    \mathrm{supp}(f)=\left\{F\in 2^{[n]}:f(F)\neq0\right\}\text{ and }\mathrm{supp}_+(f)=\left\{F\in 2^{[n]}:f(F)>0\right\}\,.
\end{equation*}
    Null $t$-designs are also known as functions of \emph{strength $t$}
    or \emph{$0$-configurations}, see~\cites{deza1982vector,deza1983functions}. A null $t$-design is \emph{$k$-uniform} if $f(F)\neq0$ implies $\lvert F\rvert=k$, that is, $\mathrm{supp}(f)\subseteq \binom{[n]}{k}$. By definition, if $f$ is a $k$-uniform null $t$-design, then it is a $k$-uniform null $s$-design for every $s<t$.
\end{definition}

\begin{proposition}\label{prop: k-uniformcond}
    Let $f:2^{[n]}\rightarrow\mathbb{R}$ be such that $f(F)\neq0$ only if $\lvert F\rvert=k$. Then, $f$ is a $k$- uniform null $t$-design if and only if the balance condition (\ref{eq: t-design condition}) is satisfied for $t$-subsets, that is,
    \begin{equation*}
	\sum_{X\subseteq F\subseteq [n]} f(F)=0\,,
    \end{equation*}
    for every $X\in 2^{[n]}$ with $\lvert X\rvert=t$.
\end{proposition}

\begin{proof}
    Let $f:2^{[n]}\rightarrow\mathbb{R}$ be a function with support contained in $\binom{[n]}{k}$ satisfying \cref{eq: t-design condition} for every $X\in 2^{[n]}$ with $\lvert X\rvert=t$. Let $X$ be a subset of $[n]$ with $|X|=s<t$. 
    Refining the balance condition (\ref{eq: t-design condition}) by grouping sets $F$ according to intermediate $t$-subsets yields
    \begin{equation*}
    \binom{k-s}{t-s}\sum_{X\subseteq F\subseteq [n]} f(F)=\sum_{\substack{X\subseteq Y\subseteq[n] \\ \lvert Y\rvert=t}}\sum_{Y\subseteq  F\subseteq[n]}f(F)=0\,,
\end{equation*}
    and we conclude that $f$ is a $k$-uniform null $t$-design. The converse of the statement is trivial.
\end{proof}
We now describe the relationship between $k$-uniform null $t$-designs and binomials in $I_{n,k,t}$. 
\begin{proposition}\label{prop:isomorphism}
    There is an isomorphism of $\Z$-modules between $\ker_\Z(\mathcal{A}_{n,k,t})$ and $k$-uniform null $t$-designs with values in $\Z$.
\end{proposition}
\begin{proof}
    Let $\mathcal{D}_{n,k,t}$ be the $\Z$-module of $k$-uniform null $t$-designs with domain $2^{[n]}$ and consider
    \begin{align*}
        \varphi:\mathcal{D}_{n,k,t}&\longrightarrow\ker_\Z(\mathcal{A}_{n,k,t})\\
        f&\longmapsto (f(F): |F|=k)\,.
    \end{align*}
    This map is well-defined. Indeed, if $f:2^{[n]}\rightarrow \Z$ is a $k$-uniform null $t$-design, and $X\in 2^{[n]}$ with $|X|=t$, condition (\ref{eq: t-design condition}) guarantees that $(f(F): |F|=k)\in \ker_\Z(\mathcal{A}_{n,k,t})$, because the sum is over $k$-subsets containing $X$ and $\mathcal{A}_{n,k,t}$ has $1$ in row (indexed by) $X$ whenever $X$ is a the $k$-subset of (the index of) column $F$ and $0$ otherwise. Clearly, $\varphi$ is $\Z$-linear and it is injective. Surjectivity is a consequence of Proposition~\ref{prop: k-uniformcond}.
\end{proof}

    From the isomorphism in Proposition~\ref{prop:isomorphism}, we see that a null $t$-design $f$ corresponds to a binomial $u_f=x^{\varphi(f)_+}- x^{\varphi(f)_-}$ in the toric ideal $ I_{n,k,t}$. The degree of $u_f$ depends on the cardinality of the support of $f$, which is bounded from below by the following result.
    \begin{theorem}[{\cite[Theorem~1]{deza1983functions}}]\label{theorem1}
    Let $f$ be a not identically zero null $t$-design. Then
    \begin{equation*}
    	\lvert\mathrm{supp}_+(f)\rvert\geq 2^{t}\,.
    \end{equation*}
\end{theorem}

\begin{corollary}
    Every polynomial in $I_{n,k,t}$ has degree at least $2^t$, with equality for squarefree binomials. 
\end{corollary}
\begin{proof}
    Since $f$ is a $k$-uniform null $t$-design if and only if $-f$ is a $k$-uniform null $t$-design, the previous theorem implies that $f$ is nonzero on at least $2^{t+1}$ subsets of $[n]$ (\cite[Theorem~1]{frankl1983number}):
\begin{equation*}
	\lvert\mathrm{supp}(f)\rvert \geq 2^{t+1}.
\end{equation*}
    The binomial $u_f$ has degree 
    \[ \deg(u_f) \geq \max\{\lvert\mathrm{supp}_+(f)\rvert, \lvert\mathrm{supp}_+(-f)\rvert \}  \geq 2^t.\] 
    Moreover, if $u_f$ is squarefree, the first bound is attained as 
    \[ \deg(u_f)= \frac{1}{2}\lvert\supp(f)\rvert =\lvert\supp_+(f)\rvert= \lvert\supp_+(-f)\rvert \geq 2^t. \qedhere \]
\end{proof} 

    We recall that a polytope is \emph{$s$-neighborly} if every set of $s$ or fewer vertices forms a face.
\begin{theorem} \label{theorem:neighborly}
     The polytope $\mathcal{P}_{n,k,t}$ is $(2^{t}-1)$-neighborly. 
\end{theorem}
\begin{proof}
    Let $F$ be a set of vertices of $\mathcal{P}_{n,k,t}$ that does not form a face. Then there exists a vertex $\mathbf{a}\not\in F$ of $\mathcal{P}_{n,k,t}$, such that some rational convex combination of vertices in $F$ coincides with a rational convex combination of vertices in $F \cup \{\mathbf{a}\}$, where $\mathbf{a}$ appears with nonzero coefficient. Equivalently, there exists a vector $v\in \ker(\mathcal{A}_{n,k,t})$ whose positive part $v_+$ corresponds to an integer linear combination of columns of $\mathcal{A}_{n,k,t}$ indexed by $F$, while the negative part $v_-$ corresponds to an integer linear combination of columns indexed by $F\cup \{\mathbf{a}\}$, with the coefficient of $\mathbf{a}$ nonzero. The associated binomial $x^{v_+}-x^{v_-}\in I_{n,k,t}$ verifies $\supp(v_+)\subseteq F$
    . Therefore, by Theorem~\ref{theorem1} we obtain $|F| \geq 2^t$ and any set of $2^t-1$ vertices or less is a face of $\mathcal{P}_{n,k,t}$.
\end{proof}
\begin{remark}
    A complete description of the faces of $\mathcal{P}_{n,k,t}$ through supporting hyperplanes appears to be complicated. However, if $2k<n$, it can be shown that every set of vertices $\mathbf{v}_1,\dots,\mathbf{v}_{\ell}$  with $\ell<2^t$ is contained in a proper face by providing a supporting hyperplane.  Let $\mathbf{v}_1,\dots,\mathbf{v}_{\ell}$ be vertices of $\mathcal{P}_{n,k,t}$ and $F_1,\dots,F_{\ell}$ be the corresponding $k$-subsets. Consider the set
    \begin{equation*}
        \mathcal{T}=\{T\in 2^{[n]}: \lvert T\rvert=t , \ \ T\subseteq F_i \text{ for some } i=1,\dots,\ell\}\,.
    \end{equation*}
    Let $H_{\mathcal{T}}$ be the affine hyperplane cut out by the equation
    \begin{equation*}
    \sum_{T\in\mathcal{T}}y_T - \binom{k}{t} =0\,.
    \end{equation*}
    Clearly, the vectors $\mathbf{v}_1,\dots,\mathbf{v}_{\ell}$ belong to $H_{\mathcal{T}}$ and each vertex of $\mathcal{P}_{n,k,t}$ is contained in the half-space $H_{\mathcal{T}}\leq~0$.
    Since $2k<n$ we have $2^t\binom{k}{t}<\binom{n}{k}$.    
    The combinatorial meaning of the inequality is the following: any $t$-subset of $k$ chosen elements can be completed to a $k$-subset of $[n]$ by picking $k-t$ elements among the remaining ones. 
    This, together with $\ell < 2^t$, implies:
    \begin{equation*}
        \lvert\mathcal{T}\rvert\leq \ell \binom{k}{t}\leq 2^t\binom{k}{t}<\binom{n}{k}.
    \end{equation*}
    It follows that there exists a set $X\notin \mathcal{T}$ of size $t$ and a vertex $\mathbf{v}$ of $\mathcal{P}_{n,k,t}$ that does not lie in $H_{\mathcal{T}}$. 
\end{remark}

    Each $k$-uniform null $t$-design can be equivalently expressed as a squarefree homogeneous polynomial of degree $k$ in $n$ variables whose coefficients are the values of $f$ on the corresponding sets of cardinality $k$. From this formulation it is easier to express a minimal set of generators for the $\Z$-module of $k$-uniform null $t$-designs. The following result was firstly proven by Graver and Jurkat in~\cite[Theorem 4.2]{graver1973module}, here we refer to the formulation of Graham, Li, and Li.
\begin{theorem}[{\cite[Theorem 1]{graham1980structure}}]\label{theorem:graham}
    The $\Z$-module of $k$-uniform null $t$-designs with domain $2^{[n]}$ and values in $\Z$ is generated by the collection of polynomials of the form
    \begin{equation*}
        p(x_1,\dots,x_n)=(x_{i_1}-x_{i_2})(x_{i_3}-x_{i_4})\dots(x_{i_{2t+1}}-x_{i_{2t+2}})x_{i_{2t+3}}\dots x_{i_{k+t+1}}\,.
    \end{equation*}
\end{theorem}
    Designs corresponding to the polynomials in Theorem~\ref{theorem:graham} are commonly known as $(t,k)$-\emph{pods} and have positive support of cardinality $2^t$. This translates into the following corollary. 

\begin{corollary}
    Theorem~\ref{theorem:neighborly} is optimal, i.e., the polytope $\mathcal{P}_{n,k,t}$ is not $2^t$-neighbourly.
\end{corollary}

\subsection{Specht polynomials and octahedra}\label{ss:specht}

    In this subsection we discuss the interpretation of $(t,k)$-pods as Specht polynomials arising from Young tableaux. This follows from the independent versions of  \Cref{theorem:graham} given by Spink in terms of orientations of hyperoctahedra and by Morita, Wachi and Watanabe~\cite[Theorem 22]{MWW2009} in terms of Specht polynomials (see also~\cite[Section 9.3]{lefschetzbook}).
    
\begin{theorem}[{\cite[Theorem 2.4]{S2018}}]\label{theorem:spink}
    For all $n,k,t$, the kernel of $\A_{n,k,t}$ is generated as a $\Z$-module by binomials arising from the orientation of hyperoctahedra.
\end{theorem}

    To explain the connection, we briefly recall the notion of Young tableaux. A sequence of integers $\lambda = (\lambda_1,\dots,\lambda_s)$ such that $\lambda_1\geq\cdots\geq\lambda_s\geq0$ defines a \emph{partition} of $n = \lambda_1 + \dots + \lambda_s$.
    The associated \emph{Young diagram} $D_\lambda$ is a left-justified array of $s$ rows of boxes stacked on top of each other where the $i$-th row of $D_\lambda$ has $\lambda_i$ boxes.
    Filling $D_\lambda$ by assigning a value to each box yields a \emph{Young tableau} of shape $\lambda$ and size $n$. 
    Young tableaux are ubiquitous in the theory of representations as partitions of $n$ and they are in bijection with the irreducible characters of $\mathfrak{S}_n$ over $\mathbb{C}$. Bases of such irreducible representations are described by \emph{standard} Young tableaux of size $n$, which are tableaux whose entries are strictly increasing along every row and every column. 
    We refer to the standard reference \cite{tableaux_fulton_1997} for a detailed treatment of these connections.
        
    To a standard Young tableau $T$ of shape $\lambda$ and size $n$ we associate its \emph{Specht polynomial}. This is a homogeneous polynomial $F_T \in \K[x_1,\dots,x_n]$ defined via Vandermonde determinants.
    Although the general definition is quite technical, it is significantly simplified in our context.
    Suppose $T$ has shape $\lambda = (n-s, s)$, and let $i_1,\dots,i_{n-s}$ and $j_1,\dots,j_s$ be the fillings of the first and second row, respectively.
    Then the Specht polynomial of $T$ is
    \begin{equation*}
        F_T = \prod_{k=1}^s (x_{i_k} - x_{j_k}).
    \end{equation*}
    If $A = \K[x_1,\dots,x_n]/(x_1^2,\dots,x_n^2)$, the Specht polynomials associated to a standard Young tableaux of shape $(n-s,s)$ span the kernel of the map (see \cite[Section 9.3]{lefschetzbook})
    \begin{equation*}
        \nabla \colon A \to A, \quad \nabla(f) = \frac{\partial f}{\partial x_1} + \dots + \frac{\partial f}{\partial x_n}.
    \end{equation*}
    The matrix representing the map $\nabla^{k-t}: A_{k} \to A_t$ is $(k-t)!$ times the matrix $\mathcal{A}_{n,k,t}$ that gives a monomial parametrization for the toric incidence ideal $I_{n, k, t}$ \cite[Proposition 3.7]{H2025}. 
    In particular, Specht polynomials of type $(n-s, s)$ can be seen as a set of vectors that generate the kernel of $\mathcal{A}_{n,k,k-1}$.
    Note that polynomials of the form $F_T$ can be characterized by a set of pairs such that these are the only pairs that do not appear together in the monomials of $F_T$. In other words, the monomials in the support of $F_T$ correspond exactly to the facets of a crosspolytope. In~\cite{S2018} the setting where $t \neq k - 1$ is also considered, and in this case Spink calls the corresponding complexes hyperoctahedra. These complexes can be obtained from crosspolytopes by coning over a simplex. 


From now on, we denote by $J_{n,k,t}$ the \emph{octahedral ideal} generated by the binomials corresponding to hyperoctahedra. By Theorem~\ref{theorem:spink} we know that the kernel of $\mathcal{A}_{n,k,t}$ is generated by the hyperoctahedra and so, by~\cite[Lemma~7.6]{MillerSturmfels2005}, we obtain
\begin{equation*}
    I_{n,k,t}=J_{n,k,t}:(x_1\cdots x_n)^{\infty},
\end{equation*}
geometrically this translates into $\V(I_{n,k,t})=\overline{\V(J_{n,k,t})\cap (\K^*)^n}$. In other words, the varieties defined by the two ideals coincide up to connected components contained in the union of the coordinate hyperplanes. This leads us to the following question.
\begin{question} \label{question: dimension embedded components}
    Do the embedded components of $\V(J_{n,k,t})$ have lower dimension than $\V(I_{n,k,t})$? In particular, does the following equality hold for every $n,k,t$
    $$
        \deg \V(I_{n,k,t}) = \deg \V(J_{n,k,t})?
    $$
\end{question}
A positive answer to \cref{question: dimension embedded components} would imply that $J_{n,k,t}$ captures most of the geometric information of the toric variety defined by $I_{n,k,t}$.   

\section{Balanced pseudomanifolds} \label{sec: balanced manifolds}

We now interpret toric incidence ideals from a topological point of view. The key insight is the fact that the monomial parametrization of $I_{n,k,k-1}$ can be viewed as a (signless) boundary map of the skeleton of a simplex. In particular, if $\Delta$ is a simplicial complex of dimension $k-1$ whose top boundary map has rows that can be multiplied by constants in order to obtain a submatrix of $\A_{n,k,t}$, then homology cycles of $\Delta$ correspond to elements of $I_{n,k,t}$.

\begin{example}\label{ex:bipartiteoctahedron}
Consider the octahedron $\Delta$ from~\cref{figure:octahedron} with partition $V_1 \sqcup V_2 \sqcup V_3$. Order the vertices of $\Delta$ such that $u < v$ whenever $u \in V_i$, $v \in V_j$ and $i < j$. Under this ordering, the last boundary map $\partial$ of $\Delta$ is the following matrix:  
$$
\small
   \kbordermatrix{
     & 135 & 145 & 136 & 146 & 235 & 245 & 236 & 246\\
     15     &-1&-1&0&0&0&0&0&0\\
     16     &0&0&-1&-1&0&0&0&0\\
     13     &1&0&1&0&0&0&0&0\\
     14     &0&1&0&1&0&0&0&0\\
     25     &0&0&0&0&-1&-1&0&0\\
     26     &0&0&0&0&0&0&-1&-1\\
     23     &0&0&0&0&1&0&1&0\\
     24     &0&0&0&0&0&1&0&1\\
     35     &-1&0&0&0&-1&0&0&0\\
     36     &0&0&-1&0&0&0&-1&0\\
     45     &0&-1&0&0&0&-1&0&0\\
     46     &0&0&0&-1&0&0&0&-1}
$$
Note that each row of $\partial$ either only has nonzero positive entries, or nonzero negative entries. In particular, multiplying every row with negative entries by $-1$ gives a submatrix of $\A_{n,k,t}$. 
\end{example}

Finding triangulations of (pseudo)manifolds with a bipartite facet-ridge graph is a problem addressed both in algebra~\cite[Theorem~4.6]{DN2024} and in combinatorics. The following statement from~\cite{KN2016} is a variation of an earlier result due to Joswig \cite[Proposition~11]{J2002} (see also~\cite[p. 22]{KNN2016}), and it is where the procedure from~\cref{ex:bipartiteoctahedron} was first described.

\begin{lemma}[{\cite[Lemma 6.2]{KN2016}}]\label{l:bipartitedual}
    Let $\Delta$ be a balanced $d$--dimensional simplicial complex with $d \geq 2$. If $\Delta$ is a normal pseudomanifold without boundary, then $\Delta$ is orientable if and only if the facet-ridge graph of $\Delta$ is bipartite. 
\end{lemma}

\cref{l:bipartitedual} explains the octahedral quartics from \cref{ex:octahedralquartics}: monomials in these quartics correspond to facets of an octahedra. The octahedron is the simplest balanced triangulation of a sphere, hence the quartics coming from orientations are in the ideal $I_{n,k,t}$ also after increasing the dimension of the octahedra and varying $n,k,t$ accordingly. We can show the following.

\begin{theorem}\label{t:topology}
    Let $k \geq 3$ and let $\Delta = \tuple{\sigma_1, \dots, \sigma_s}$ be a balanced orientable normal $(k-1)$--pseudomanifold without boundary on $n$ vertices,  and orientation $\varepsilon = \varepsilon_1 \sigma_1 + \dots + \varepsilon_s \sigma_s$. Then 
    \begin{equation} \label{eq: binomial from orientation}
         \prod_{\substack{i \in [n] \\ \varepsilon_i = 1}} c_{\sigma_i} - \prod_{\substack{j \in [n] \\ \varepsilon_j = -1}} c_{\sigma_j} \in I_{n,k,k-1}.
    \end{equation}
    Moreover, the binomials above are in the Graver basis of $I_{n, k, k - 1}$.
\end{theorem}

\begin{proof}
    By~\cref{l:bipartitedual} the facet-ridge graph of $\Delta$ is bipartite. Since $\Delta$ is a pseudomanifold without boundary, we know every ridge is contained in exactly $2$ facets. Given a partition $V_1 \sqcup \dots \sqcup V_k$ of the vertices of $\Delta$, order the vertices such that $u < v$ whenever $u \in V_i$ and $v \in V_j$ where $i < j$. Under this ordering, the sign of a nonzero entry corresponding to a facet $F$ and a ridge $\sigma \subset F$ in the top boundary map of $\Delta$ depend only on the parity of $i$ where $u = F \setminus \sigma$ and $u \in V_i$. This implies we may multiply the rows of the top boundary map $\partial$ of $\Delta$ to remove negative signs and obtain a submatrix of $\A_{n,k,k-1}$. More specifically, there exists an invertible diagonal matrix $D$ such that $D\partial$ is a submatrix of $\A_{n,k,k-1}$. Since $\Delta$ is orientable, the map $\partial$ has a kernel generated by $\varepsilon$. Hence $D\partial$ also has a $1$-dimensional kernel generated by $\varepsilon$. The first part of the  statement then follows. 
    
    For the final part, assume for the sake of contradiction that the binomial is not primitive. Then there exists $\varepsilon'$ and another binomial 
    $$
        m' = \prod_{\substack{i \in [n] \\ \varepsilon_i' = 1}} c_{\sigma_i} - \prod_{\substack{j \in [n] \\ \varepsilon_j' = -1}} c_{\sigma_j} \in I_{n,k,k-1},
        $$
        such that $\varepsilon'_+ < \varepsilon_+$ and $\varepsilon'_- < \varepsilon_-$.
        As $\Delta$ is orientable, we know $\tilde H_{k-1}(\Delta; \Z) = \Z$, and since by the arguments in the first part $\varepsilon'$ is also a nonzero element in $\tilde H_{k-1}(\Delta; \Z)$, we conclude $\varepsilon'$ is a constant multiple of $\varepsilon$, that is a contradiction.
\end{proof}


An interesting consequence of~\cref{t:topology} is that we can use results from the theory of triangulations of balanced pseudomanifolds in order to find new elements of $I_{n, k, k-1}$ by applying certain operations on binomials in $I_{n', k, k - 1}$, where $n' < n$.

\begin{example}
    In~\cite{IKN2017}, Izmestiev, Klee and Novik study the balanced versions of bistellar flips, which they call \emph{cross-flips}. We describe a new type of binomial in the Graver basis of $I_{9, 3, 2}$ based on applying one of their cross-flips to the octahedral quartics of $I_{6,3,2}$.
    
    Let $m =  c_{136}c_{246}c_{145}c_{235}  - c_{146}c_{236}c_{245}c_{135}$ be the octahedral quartic from~\eqref{eq: octahedral quartics}. Applying~\cref{t:topology} and the first move from~\cite[Fig. 1]{IKN2017} to the facet $\{1,4,5\}$ of the  boundary of the octahedron from~\Cref{figure:octahedron}, we get the primitive binomial 
    \[
        c_{146} c_{236} c_{135} c_{245} c_{678} c_{179} c_{389} - c_{789} c_{167} c_{368} c_{139} c_{246} c_{145} c_{235} \in I_{9, 3,2}.
    \]
    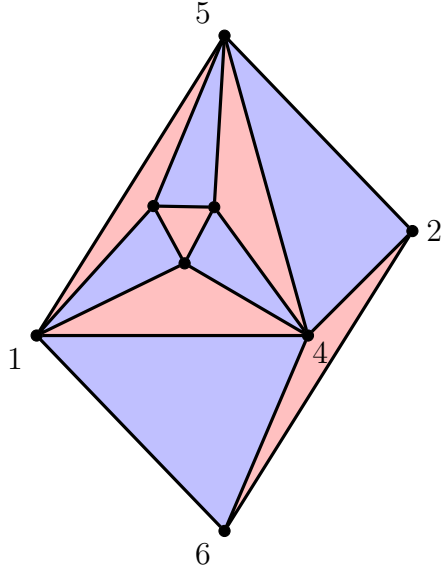
\begin{figure}[!h]
        \centering
        \begin{tikzpicture}[scale=1.8]
            \node[draw,circle,fill=black,scale=0.2,label={[label distance=-1pt]210:$1$}] at (xyz spherical cs:radius=1.41,latitude=135,longitude=90) (1) {$1$};
            \node[draw,circle,fill=black,scale=0.2,label={[label distance=-1pt]0:$2$}] at (xyz spherical cs:radius=1.41,latitude=-45,longitude=90) (3) {$3$};
            \node[draw,circle,fill=black,scale=0.2,label={[label distance=-5pt]-30:$4$}] at (xyz spherical cs:radius=1.41,latitude=45,longitude=90) (4) {$4$};
            \node[draw,circle,fill=black,scale=0.2,label={[label distance=-1pt]150:$5$}] at (xyz spherical cs:radius=1.82) (5) {$5$};
            \node[draw,circle,fill=black,scale=0.2,label={[label distance=-1pt]210:$6$}] at (xyz spherical cs:radius=1.82,latitude=180) (6) {$6$};
            
            \draw[very thick] (3) -- (4) -- (1);
            \draw[very thick] (1) -- (5);
            \draw[very thick] (3) -- (5) -- (4);
            \draw[very thick] (1) -- (6);
            \draw[very thick] (3) -- (6) -- (4);

            \node[draw,circle,fill=black,scale=0.2] at (xyz spherical cs:radius=0.88,latitude=-26,longitude=-58) (cf1) {$1$};
            \node[draw,circle,fill=black,scale=0.2] at (xyz spherical cs:radius=0.88,latitude=32,longitude=8) (cf2) {$2$};
            \node[draw,circle,fill=black,scale=0.2] at (xyz spherical cs:radius=0.88,latitude=60,longitude=0) (cf3) {$3$};

            \draw[very thick] (cf1) -- (cf2) -- (cf3) -- (cf1);
            \draw[very thick] (cf1) -- (1) -- (cf3);
            \draw[very thick] (cf1) -- (5) -- (cf2);
            \draw[very thick] (cf2) -- (4) -- (cf3);

            \begin{scope}[on background layer]
                \path[fill=blue,opacity=0.25] (3.center) to (4.center) to (5.center) to (3.center);
                \path[fill=red,opacity=0.25] (3.center) to (4.center) to (6.center) to (3.center);
                \path[fill=blue,opacity=0.25] (4.center) to (1.center) to (6.center) to (4.center);
                \path[fill=red,opacity=0.25] (cf1.center) to (cf2.center) to (cf3.center) to (cf1.center);
                \path[fill=red,opacity=0.25] (cf2.center) to (4.center) to (5.center) to (cf2.center);
                \path[fill=red,opacity=0.25] (cf1.center) to (1.center) to (5.center) to (cf1.center);
                \path[fill=red,opacity=0.25] (cf3.center) to (4.center) to (1.center) to (cf3.center);
                \path[fill=blue,opacity=0.25] (cf3.center) to (1.center) to (cf1.center) to (cf3.center);
                \path[fill=blue,opacity=0.25] (cf3.center) to (4.center) to (cf2.center) to (cf3.center);
                \path[fill=blue,opacity=0.25] (cf1.center) to (cf2.center) to (5.center) to (cf1.center);
            \end{scope}
        \end{tikzpicture}
        \caption{Triangulation of a $2$-sphere obtained by applying the first cross-flip from~\cite[Fig 1]{IKN2017} to the facet $\{1,4,5\}$ of the complex from~\Cref{figure:octahedron}.}
    \end{figure}
    Octahedral quartics obtained by different labelings induce different primitive binomials in~$I_{9,3,2}.$
\end{example}

\begin{remark}[{\bf Principal symmetric ideals}]
    In~\cite{HSS2025}, Harada, Seceleanu and \c{S}ega introduced a class of symmetric ideals called~\emph{principal symmetric ideals}, which are ideals of the form $(\sigma \cdot f \st \sigma \in \mathfrak{S}_n) \subseteq R = \K[x_1,\dots, x_n]$, where $\mathfrak{S}_n$ is the symmetric group on $n$ elements, and the action of $\mathfrak{S}_n$ on $R$ is given by permuting variables. We note that the ideal $J_{n,k,t}$ is not principal symmetric in the sense of~\cite{HSS2025}, since it is inside a polynomial ring on $\binom{n}{k}$ variables, and it is not invariant under the natural action of the symmetric group on $\binom{n}{k}$ elements. There is, however, a natural action of the symmetric group on $n$ elements on the ring of $J_{n,k,t}$ that sends $x_A$ to $x_{\sigma A}$, where $\sigma A = \{\sigma(i) \st i \in A\}$. Under this action, the ideal $J_{n,k,t}$ is easily seen to be principal and symmetric, as it is generated by a single binomial $m$ and other binomials of the form $\sigma \cdot m$.  
\end{remark}

\section{Three point functions} \label{sec: three point functions}
    In this final section, we return to our original motivation and focus on studying the algebraic relations among three-point functions. In our setup, we consider $n$ massless particles with momenta vectors $p_1,\dots,p_n$ in $\R^d$. We denote by $p_{i,j}$ the pairwise products $p_i\cdot p_j$ with respect to the Lorentzian inner product and collect them in an $n\times n$ symmetric matrix $P=(p_{i,j})$, that has zeros on the diagonal. We focus on the following simplified form of three point functions
\begin{equation*}
    C_{i,j,k}:=\frac{1}{p_{i,j}p_{j,k}p_{i,k}}\text{ for }1\leq i< j<k\leq n\,.
\end{equation*}
    The algebraic relations among these functions are the same as the relations among their reciprocals $c_{i,j,k}=p_{i,j}p_{j,k}p_{ik}$, generating the ideal $\tilde I_{n,d}$ of $\C[c_{i,j,k}:1\leq i< j<k\leq n]$.
    If the dimension of the ambient space $d$ is at least the number of particles $n$, there are no algebraic relations among the $p_{i,j}$'s, that we treat as independent variables. Then the ideal $\tilde I_{n,d}$ coincides with the incidence toric ideal $I_{n,3,2}$ in \Cref{def:inctorideal}. In general, we have the following result.
\begin{theorem}\cite[Theorem 2.5]{el2024gram}\label{thm:algebraicrelations}
    The algebraic relations among the $p_{i,j}$'s are given by the $(d+1)$-minors of $P$. 
\end{theorem}

\noindent
    The ideals $\tilde I_{n,d}$ form a stabilizing descending chain
\begin{equation*}
    \tilde I_{n,1}\supseteq \tilde I_{n,2}\supseteq\cdots \supseteq \tilde I_{n,n}=\tilde I_{n,n+1}=\cdots = I_{n,3,2}.
\end{equation*}
    The ideal $\tilde I_{n,n}$ has been widely discussed in previous sections. Here we focus on the smallest non-toric case, which is $\tilde I_{n,n-1}$. This is obtained by adding the determinant of the symmetric matrix $P$ to $\tilde I_{n,n}$. However, the determinant is an element of $\C[p_{i,j}]$, and it may not be contained in the subring $\C[c_{i,j,k}]$. Therefore, we must determine to what extent $\det(P)$ can be expressed in terms of $c_{i,j,k}$'s. To be more precise, our goal is to answer the following question.
\begin{question}\label{question:main}
    Consider the symmetric matrix $P$ with zero on the diagonal. 
    Does there exist an element $q\in \mathbb{K}[p_{i,j}][x]$ such that $q(\det(P_n))\in \mathbb{K}[c_{i,j,k}]$?
\end{question}

We denote by $(G_n,\cdot)$ the free abelian group generated by the $p_{i,j}$'s and by $C_n$ the subgroup of $G_n$ generated by the $c_{i,j,k}$'s.
Let $H_n\subseteq S_n$ be the set of derangements of $n$ elements, i.e., all the permutations in $S_n$ with no fixed points. We consider the map
\[ \varphi_n:H_n\rightarrow G_n, \quad \varphi_n(\sigma)=p_{1,\sigma(1)}p_{2,\sigma(2)}\dots p_{n,\sigma(n)}. \]
We simply write $\varphi$ when the domain is clear from the context.
\begin{proposition}
    The map $\varphi$ is not injective. In particular, we have $\lvert\varphi^{-1}(\varphi(\sigma))\rvert=2^{t(\sigma)-s(\sigma)}$, where $t(\sigma)$ is the number of cycles appearing in the cycle decomposition of $\sigma$ and $s(\sigma)$ is the number of transpositions in such decomposition.

\end{proposition}
\begin{proof}
    Any $n$-cycle $\sigma$ and its inverse have the same image. It is enough to show it for $\sigma=(1 \ 2 \ \cdots \  n )$. 
   If $\varphi(\tau)=\varphi(\sigma)=p_{1,2}p_{2,3}\cdots p_{n,1}$ we have $\tau(1)=2$ or $\tau(2)=1$, as $p_{1,2}=p_{2,1}$. The first case, as a domino effect, implies $\tau(2)=3 ,\dots, \tau(n)=1$, that is, $\tau =\sigma$. Similarly, in the second case $\tau=\sigma^{-1}$.
   Since $\varphi$ commutes with respect to the product of disjoint cycles $\sigma_1,\dots, \sigma_t$ we get:
   \[ \varphi^{-1}(\varphi(\sigma_1\cdots\sigma_t)) = \{ (\sigma_1\cdots \sigma_t)^\varepsilon: \varepsilon\in \{\pm 1\}^t \}.
   \]
   Considering that a $2$-cycle is self-inverse, we conclude that $\lvert \varphi^{-1}(\varphi(\sigma))\rvert = 2^{t(\sigma)- s(\sigma)}$.
\end{proof}
    Since the sum in the Leibniz expansion of $\det(P_n)$ runs only over $H_n$, a strategy to answer Question~\ref{question:main} is to understand for which $n\in\mathbb{N}$ the image of $\varphi$ is in $C_n$. We prove it for $n=3k$ and use this to attack the remaining cases. The following lemma is a direct computation.
\begin{lemma} \label{lemma: transposition 1&2}
    For distinct indices $i,j,k,s,t\in [n]$ we have the following relations

        $$
            p_{k,i} p_{j,s} = p_{k,j} p_{i,s} \frac{c_{k,i,t}c_{j,s,t}}{c_{k,j,t} c_{i,s,t}} \qand p_{i,j} p_{i,j} =p_{k,s} p_{k,s} \frac{c_{i,j,k} c_{i,j,s}}{c_{i,k,s} c_{j,k,s}}.
        $$
\end{lemma}
    As next step, we show that images of derangements belong to the same left coset of $C_n$ in $G_n$.
    

\begin{proposition}\label{proposition: n-1cycle and transitivity}
    If $n\geq 5$, for every derangement $\sigma\in H_n$ with $t(\sigma)\geq 2$ there exists $\bar\sigma\in H_n$ with $t(\bar\sigma)=t(\sigma)-1$ such that $\varphi(\sigma)\in\varphi(\bar\sigma)C_n$. In particular, for any $\sigma_1,\sigma_2\in H_n$ we have \[\varphi(\sigma_1)\in\varphi(\sigma_2)C_n.\] 
\end{proposition}
\begin{proof}
     Since $t(\sigma)\geq 2$, there are distinct $i,j$ belonging to two different cycles in the decomposition of $\sigma$. Then, $t((i,j)\circ\sigma)=t(\sigma)-1$. By Lemma~\ref{lemma: transposition 1&2} for any $\ell\neq i$, $j$, $\sigma^{-1}(i)$, $\sigma^{-1}(j)$
     \[ \varphi((i,j)\circ\sigma)=\varphi(\sigma)\frac{c_{\sigma^{-1}(i),j,\ell} \ c_{\sigma^{-1}(j),i,\ell}}{c_{\sigma^{-1}(i),i,\ell} \ c_{j,\sigma^{-1}(j),\ell}}\in C_n,\]
    and we conclude that $\varphi(\sigma)\in\varphi(\bar\sigma)C_n$. The second part follows using Lemma~\ref{lemma: transposition 1&2}.
\end{proof}

\begin{proposition}\label{proposition:0=3}
    Let $n$ be an integer divisible by $3$ and let $\sigma\in H_n$. Then, $\varphi(\sigma)\in C_n$.
\end{proposition}
\begin{proof}
    Since $n$ is divisible by $3$, we have $\sigma=(1 \ 2 \ 3)(4 \ 5 \ 6)\cdots(n-2 \ n-1 \ n)\in H_n$. Clearly, $\varphi(\sigma)=c_{1,2,3}c_{4,5,6}\cdots c_{n-2,n-1,n}\in C_n$. We conclude by applying Proposition~\ref{proposition: n-1cycle and transitivity}.
\end{proof}

\begin{lemma}\label{lemma:product}
    Let $n>2$ be a positive odd integer such that $n\equiv 2\pmod 3$. Then,
    $$\prod_{1\leq i<j\leq n}p_{i,j}\in p_{1,3}p_{2,3}p_{2,4}p_{1,4}\cdot C_{n}.$$
\end{lemma}
\begin{proof}
     Let $q$ be a reduced word in $G_n$ such that $\prod_{1\leq i<j\leq n}p_{i,j}\in qC_n$. Since $\binom{n}{2}\equiv 1\pmod{3}$, we have that the $\mathrm{lt}(q)\equiv 1\pmod{3}$, where $\mathrm{lt}(q)$ is the length of $q$. Since $n$ is odd, every index $i\in[n]$ appears an even number of times in $\prod_{1\leq i<j\leq n}p_{i,j}$, and therefore also in $q$. This implies that $\mathrm{lt}(q)>1$. Moreover, if $\mathrm{lt}(q)\geq 7$, since every index $i\in[n]$ appears an even number of times, there exist $\sigma_1,\dots,\sigma_{\ell}$ such that $q=\varphi(\sigma_1)\cdots\varphi(\sigma_{\ell})$.  Applying Lemma~\ref{lemma: transposition 1&2} multiple times gives that for each $i\in[\ell]$,  $\sigma_i$ is either a transposition or a $3$-cycle. We can assume that at least one of them is a $3$-cycle. Otherwise, we have at least $3$ transpositions appearing, and by applying Lemma~\ref{lemma: transposition 1&2} we can construct a $3$-cycle. So, without loss of generality, let $\sigma_1$ be a $3$-cycle. Then $q\in \varphi(\sigma_2)\cdots\varphi(\sigma_{\ell})C_n$ and $\mathrm{lt}(\varphi(\sigma_2)\cdots\varphi(\sigma_{\ell}))<\mathrm{lt}(q)$, and we conclude.
\end{proof}

\begin{proposition}\label{proposition:2=3}
    Let $n>2$ be an odd integer such that $n\equiv 2\pmod 3$ and let $\sigma\in H_n$. Then: \[\varphi(\sigma)\prod_{1\leq i<j\leq n}p_{i,j}\in C_n.\]
\end{proposition}
\begin{proof}
    Let $\sigma\in H_n$ be a derangement. Then, by applying Lemma~\ref{lemma: transposition 1&2} and up to reordering the indices, we have that
    $\varphi(\sigma)\prod_{1\leq i<j\leq n}p_{i,j}\in C_n$ if and only if $$\varphi( (1,2)(3,\dots,n))\prod_{1\leq i<j\leq n}p_{i,j}\in C_n\,.$$ By Proposition~\ref{proposition:0=3}, we obtain that $\varphi((3,\dots,n))\in C_n$. By Lemma~\ref{lemma:product}, we have that
    \begin{equation*}
        \begin{split}
            \varphi( (1,2))\prod_{1\leq i<j\leq n}p_{i,j}\in C_n&\iff \varphi((1,2))p_{1,3}p_{2,3}p_{2,4}p_{1,4}C_{n}\in C_n\\
            &\iff p_{1,2}p_{1,2}p_{1,3}p_{2,3}p_{2,4}p_{1,4}C_{n},
        \end{split}
    \end{equation*}
    and so we conclude since $p_{1,2} \ p_{1,2} \ p_{1,3}  \ p_{2,3} \ p_{2,4} \ p_{1,4}=c_{1,2,3} \ c_{1,2,4}\in C_n$.
\end{proof}
\begin{proposition}\label{proposition2:2=3}
    Let $n>2$ be an integer such that $n\equiv 2\pmod 3$ and let $\sigma_1,\sigma_2,\sigma_3\in H_n$. Then, $\varphi(\sigma_1)\varphi(\sigma_2)\varphi(\sigma_3)\in C_n$
\end{proposition}
\begin{proof}
    By Proposition~\ref{proposition: n-1cycle and transitivity} and Proposition~\ref{proposition:0=3} we have that for any $\sigma\in H_n$
    \[ \varphi(\sigma)\in \varphi((1,2))C_n=\varphi((1,3))C_n=\varphi((2,3))C_n. \]
    Then
    \begin{equation*}
        \varphi(\sigma_1)\varphi(\sigma_2)\varphi(\sigma_3)\in \varphi((1,2))\varphi((1,3))\varphi((2,3))C_n.
    \end{equation*}
    Moreover,
    \begin{equation*}
        \varphi((1,2))\varphi((1,3))\varphi((2,3))=p_{1,2}^2p_{1,3}^2p_{2,3}^2=c_{1,2,3}^2,
    \end{equation*}
    and so $\varphi(\sigma_1)\varphi(\sigma_2)\varphi(\sigma_3)\in C_n$ for all $\sigma_1,\sigma_2,\sigma_3\in H_n$.
\end{proof}
\begin{proposition}\label{proposition:1=3}
    Let $n>4$ be an integer such that $n\equiv 1\pmod 3$ and let $\sigma_1,\sigma_2,\sigma_3\in H_n$. Then, $\varphi(\sigma_1)\varphi(\sigma_2)\varphi(\sigma_3)\in C_n$.
\end{proposition}
\begin{proof}
    By Proposition~\ref{proposition: n-1cycle and transitivity} and Proposition~\ref{proposition:0=3} we have that 
    $$\varphi(\sigma)\in \varphi((1,2)(3,4))C_n=\varphi((1,3)(2,4))C_n=\varphi((1,4)(2,3))C_n$$
    for any $\sigma\in H_n$. Then,
    \begin{equation*}
        \varphi(\sigma_1)\varphi(\sigma_2)\varphi(\sigma_3)\in \varphi((1,2)(3,4))\varphi((1,3)(2,4))\varphi((1,4)(2,3))C_n.
    \end{equation*}
    Moreover,
    \begin{equation*}
        \varphi((1,2)(3,4))\varphi((1,3)(2,4))\varphi((1,4)(2,3))=p_{1,2}^2p_{3,4}^2p_{1,3}^2p_{2,4}^2p_{1,4}^2p_{2,3}^2=c_{1,2,3}c_{1,2,4}c_{1,3,4}c_{2,3,4},
    \end{equation*}
    and so $\varphi(\sigma_1)\varphi(\sigma_2)\varphi(\sigma_3)\in C_n$ for all $\sigma_1,\sigma_2,\sigma_3\in H_n$.
\end{proof}
\begin{theorem}\label{theorem:detnmod3}
    Let $n\geq 3$ and $n\neq 4$. Then, 
    \begin{enumerate}
        \item $\det(P_n)\in\mathbb{K}[C_{n}]$ if $n\equiv0\pmod3$,
        \item $\det(P_n)^3\in\mathbb{K}[C_{n}]$,
        \item $\det(P_n)\prod_{1\leq i<j\leq n}p_{i,j}\in\mathbb{K}[C_{n}]$ if $n\equiv 5\pmod6$.
    \end{enumerate}
\end{theorem}
\begin{proof}
    Follows from Proposition~\ref{proposition:0=3}, Proposition~\ref{proposition:2=3}, Proposition~\ref{proposition2:2=3}, and Proposition~\ref{proposition:1=3}.
\end{proof}
The previous theorem allows us to prove that, if $n\equiv0\pmod3$, the ideal $\tilde I_{n,n-1}$ is generated, up to saturation, by the elements of $\tilde I_{n,n}$ and an additional polynomial in $\K[c_{i,j,k}]$.
\begin{proposition}
    For $n\equiv 0\pmod3$, there exists $f\in\K[c_{i,j,k}]$ such that
    \begin{equation*}
        \tilde I_{n,n-1}=\tilde I_{n,n}+ \left(f:\left( \prod c_{i,j,k}\right)^{\infty} \right)\,.
    \end{equation*}   
\end{proposition}
\begin{proof}
    By Theorem~\ref{theorem:detnmod3} there exist $f,g\in\K[c_{i,j,k}]$ coprime such that $\det(P_n)=f/g$, with $g$ a monomial. Since the ideal $\tilde I_{n,n-1}$ is the kernel of the map
    \begin{equation*}
        \varphi: \K[c_{i,j,k}]\rightarrow \K[p_{i,j}]/(\det(P)),
    \end{equation*}
    if $r$ is in $\ker(\varphi)$, there exists $h=\sum_\alpha h_\alpha\in\K[p_{i,j}]$ with $h_{\alpha}$ monomials such that
    \begin{equation*}
        \varphi(r)\varphi(g)=\varphi(f)\sum_\alpha h_\alpha.
    \end{equation*}
    We prove by induction of the number of terms in $h$ that for every $\alpha$ there exist two monomials $s_1,s_2\in\K[c_{i,j,k}]$ such that $h_{\alpha}=\varphi(s_1/s_2)$. If $h=0$, the statement is trivially true. If $h\neq0$, also $\varphi(f)h\neq0$. Let $m$ be a term of $\varphi(f)h$. This implies that there exists a term $f_i$ of $f$, a term $r_i$ of $r$, an index $\bar \alpha$, and a costant $c\in\K$ such that $\varphi(cr_i)\varphi(g)=\varphi(f_i)h_{\bar\alpha}$ and so $h_{\bar \alpha}=\varphi(cr_ig/f_i)$. Then, we get
    \begin{equation*}
        \varphi(r)\varphi(g)-\varphi(f)\varphi\left(\frac{cr_ig}{f_i}\right)=\varphi(f)\sum_{\alpha\neq\bar\alpha} h_\alpha
    \end{equation*}
    that is
    \begin{equation*}
        \varphi(f_ir-cfr_i)\varphi(g)=\varphi((f_i)f)\sum_{\alpha\neq\bar\alpha} h_\alpha.
    \end{equation*}
    Since $\sum_{\alpha\neq\bar\alpha} h_\alpha$ as one term less than $h$, we conclude by the inductive hypothesis.
    
    So, there exist $\bar r\in \tilde I_{n,n}$, $q\in\K[c_{i,j,k}]$ and $m\in\N$ such that 
    \[ r= \frac{f}{g}\frac{q}{\left(\prod c_{i,j,k}\right)^m}+\bar r\,. \]
    We conclude that $r\in \tilde I_{n,n}+ \left(f:\left( \prod c_{i,j,k}\right)^{\infty} \right)$, and so $\tilde I_{n,n-1}\subseteq \tilde I_{n,n}+ \left(f:\left( \prod c_{i,j,k}\right)^{\infty} \right)$. The other inclusion is straightforward.
\end{proof}

   This result represents the first step toward understanding the non-toric case of the ideals $\tilde I_{n,d}$ and opens new directions for further investigation.

\begin{openproblem}
    How can one construct a basis for the ideals $\tilde I_{n,d}$, when $d<n-1$?
\end{openproblem}
    
\begin{acknowledgement}
The authors would like to express their gratitude to Bernd Sturmfels for introducing the topics addressed in this work and to Gil Kalai for fruitful discussions. A special thank to the Fields Institute and the Max Planck Institute for Mathematics in the Sciences for their support in establishing the collaboration among the authors. F. S. is supported by the P500PT-222344 SNSF project.
\end{acknowledgement}

\bibliographystyle{alpha}
\bibliography{bibliography.bib}

\end{document}